\documentclass[12pt]{amsart}

\usepackage{fullpage}
\usepackage{amsthm,amsmath,amssymb}
\usepackage{graphicx}
\usepackage{float}
\usepackage{hyperref}
\usepackage{mathtools}
\usepackage{relsize}
\usepackage{ytableau}
\usepackage{cite}

 \usepackage{enumerate}

\usepackage[toc,page]{appendix}
\usepackage{lscape}
\usepackage{multirow}
\usepackage{adjustbox,expl3,etoolbox} 

\newtheorem{thm}{Theorem}[section]
\newtheorem{lem}[thm]{Lemma}
\theoremstyle{definition}

\newtheorem{cor}[thm]{Corollary}

\newtheorem{prob}[thm]{Problem}
\newcommand{\C}{{\mathbb C}}

\DeclareMathOperator\Der{Der}

\DeclareMathOperator\Ind{ind}
\DeclareMathOperator\one{\bf{1}}
\DeclareMathOperator\sym{Sym}

\DeclareMathOperator\cay{Cay}

\newcommand\id{\mathrm{id}}

\letcs\replicate{prg_replicate:nn} \newcommand*\longsum[1][1]{%
	\mathop{\textnormal{%
			\clipbox{0pt 0pt {.5\width} 0pt}{$\displaystyle\sum$}%
			\replicate{#1}{\clipbox{{.5\width} 0pt {.4\width} 0pt}{$\displaystyle\sum$}}%
			\clipbox{{.6\width} 0pt 0pt 0pt}{$\displaystyle\sum$}}}%
		}

\begin{document}	

\title[The Erdos-Ko-Rado Theorem for 2-pointwise and 2-setwise intersecting
permutations]{2-intersecting Permutations}

\author[Karen Meagher and  Bidy ]{Karen Meagher${^*}$ and  A. S. Razafimahatratra}
\address{Department of Mathematics and Statistics, University of Regina,
  Regina, Saskatchewan S4S 0A2, Canada}\email{karen.meagher@uregina.ca\\ sarobidy@uregina.ca}

\thanks{${^*}$Research supported in part by an NSERC Discovery Research Grant,
    Application No.:  RGPIN-03852-2018.}

\begin{abstract}
  In this paper we consider the Erd\H{o}s-Ko-Rado property for both
  $2$-pointwise and $2$-setwise intersecting permutations. Two
  permutations $\sigma,\tau \in \sym(n)$ are $t$-setwise intersecting
  if there exists a $t$-subset $S$ of $\{1,2,\dots,n\}$ such that
  $S^\sigma = S^\tau$. If for each $s\in S$, $s^\sigma = s^\tau$, then
  we say $\sigma$ and $\tau$ are $t$-pointwise intersecting.  We say
  that $\sym(n)$ has the $t$-setwise (resp. $t$-pointwise) intersecting
  property if for any family $\mathcal{F}$ of $t$-setwise
  (resp. $t$-pointwise) intersecting permutations, $|\mathcal{F}| \leq
  (n-t)!t!$ (resp. $|\mathcal{F}| \leq (n-t)!$). Ellis ([``Setwise intersecting
  families of permutations''. {\it Journal of Combinatorial Theory,
    Series A}, 119(4):825–849, 2012.]), proved that for $n$
  sufficiently large relative to $t$, $\sym(n)$ has the $t$-setwise
  intersecting property. Ellis also conjuctured that this result holds
  for all $n \geq t$. Ellis, Friedgut and Pilpel [Ellis, David, Ehud
  Friedgut, and Haran Pilpel. ``Intersecting families of permutations.''
  {\it Journal of the American Mathematical Society} 24(3):649-682, 2011.] also proved that for $n$ sufficiently large relative to
  $t$, $\sym(n)$ has the $t$-pointwise intersecting property. It is
  also conjectured that $\sym(n)$ has the $t$-pointwise intersecting
  propoperty for $n\geq 2t+1$. In this work, we prove these two
  conjectures for $\sym(n)$ when
  $t=2$.
\end{abstract}



\date{\today}

\maketitle

\section{Introduction}

The study of intersecting properties of finite sets is a central theme
in extremal combinatorics. A collection or family of
subsets $\mathcal{F}$ of $\{1,2,\dots,n\}$ is called \emph{intersecting} if for
any $A,B\in \mathcal{F}$, $A\cap B \neq \varnothing$. In 1961,
Erd\H{o}s, Ko and Rado proved an important result on intersecting
families of $k$-subsetes of $\{1,2,\dots,n\}$. The collection of all $k$-subsets of
$\{1,2,\dots,n\}$ is denoted by $\binom{[n]}{k}$. This result is stated as follows.

\begin{thm}[Erd\H{o}s-Ko-Rado, \cite{erdos1961intersection}]
  For any positive integer $n$ and $k$ such that $n\geq 2k$, if
  $\mathcal{F}$ is an intersecting family of $\binom{[n]}{k}$, then
  $|\mathcal{F}| \leq \binom{n-1}{k-1}$. Moreover, if $n>2k$, then
  $|\mathcal{F}| = \binom{n-1}{k-1}$ if and only if
  $\mathcal{F}= \mathcal{F}_i = \left\{ A\in \binom{[n]}{k} \mid i\in
    A \right\}$, for some $i\in \{1,2,\dots,n\}$. \label{EKR}
\end{thm}

There are several proofs and extensions of Theorem~\ref{EKR} in the
literature \cite{erdos1961intersection, katona1972simple, furedi2006proof,
  deza1983erdos, Frankl1977maximum, deza1978intersection,
  ellis2012setwise, ellis2011intersecting}.
In particular, Deza and Frankl \cite{Frankl1977maximum} extended
Theorem~\ref{EKR} for permutations. A family of permutations
$\mathcal{F} \subseteq \sym(n)$ is called $t$-pointwise intersecting if for any
$\sigma,\tau \in \mathcal{F}$, there exists
$i_1,i_2,\ldots,i_t \in \{1,2,\dots,n\}$ such that $\sigma(i_\ell) = \tau(i_\ell)$, for
all $\ell \in \{1,2,\ldots,t\}$. It is proven in~\cite{Frankl1977maximum}
that if $\mathcal{F}$ is a family of $1$-pointwise intersecting permutations of
the symmetric group $\sym(n)$, then $|\mathcal{F}| \leq (n-1)!$. In 2003, Cameron and Ku
\cite{cameron2003intersecting}, independently Larose and Malvenuto
\cite{larose2004stable}, proved that the only intersecting families of
permutations meeting the bound are cosets of a stabilizer of a
point. In 2009, Godsil and Meagher~\cite{godsil2009new} gave an algebraic
proof of this result; the work in this paper uses a similar algebraic approach.

We can also consider intersecting families of permutations from a
specific permutation group, rather than all of $\sym(n)$. For an arbitrary $t$, we say that a permutation
group $G$ has the \textsl{$t$-pointwise intersecting} property if any
family $\mathcal{F}$ of $t$-pointwise intersecting permutations, is no
larger than the maximum size of a pointwise stabilizer of a $t$-set.
Deza and Frankl's result~\cite{Frankl1977maximum} proves that
$\sym(n)$ has the $t$-ponitwise intersecting property.

A natural extension of this type of result is to consider the setwise
action of the permutations. The family of permutations $\mathcal{F}$
is called \textsl{$t$-setwise intersecting} if for any $\sigma,\tau
\in \mathcal{F}$, there exists $S\in \binom{[n]}{t}$ such that
$S^\sigma = S^{\tau}$. The stabilizer of a $t$-set is an example of a
$t$-setwise intersecting family of permutations.  We say that a
permutation group $G$ has the \textsl{$t$-setwise intersecting}
property if any family $\mathcal{F}$ of $t$-setwise intersecting
permutations, is no larger than the maximum size of a setwise
stabilizer of a $t$-set.

Note that $1$-pointwise intersecting and $1$-setwise intersecting are
equivalent, so we simply call this property intersecting.

 In particular, $\sym(n)$ has the $t$-setwise (resp. $t$-pointwise)
intersecting property if for any family $\mathcal{F}$ of $t$-setwise
(resp. $t$-pointwise) intersecting permutations, $|\mathcal{F}| \leq
(n-t)!t!$ (resp. $|\mathcal{F}| \leq (n-t)!$).

It was also conjectured in \cite{Frankl1977maximum} that
for $n$ sufficiently large with respect to $t$, a $t$-pointwise intersecting
family $\mathcal{F}$ of $\sym(n)$ is such that
$|\mathcal{F}| \leq (n-t)!$. This conjecture was proved by Ellis et
al. \cite{ellis2011intersecting} using spectral methods and the
representation of the symmetric group.  

\begin{thm}[Ellis, Friedgut, Pilpel \cite{ellis2011intersecting}]
  For $n$ sufficiently large with respect to $t$, if a family of
  permutations $\mathcal{F}$ of $\sym(n)$ is $t$-pointwise intersecting,
  then $|\mathcal{F}| \leq (n-t)!$. Moreover,
  $|\mathcal{F}| = (n-t)!$ if and only if $\mathcal{F}$ is a coset
  of the stabilizer of $t$ elements from $\{1,2,\dots,n\}$.
\end{thm}

There is a similar result for $t$-setwise intersection. It was proved
in 2011 by Ellis \cite{ellis2012setwise}.

\begin{thm}[Ellis, \cite{ellis2012setwise}]
  For $n$ sufficiently large with respect to $t$, if a family of
  permutations $\mathcal{F}$ of $\sym(n)$ is $t$-setwise intersecting,
  then $|\mathcal{F}| \leq t!(n-t)!$. Moreover,
  $|\mathcal{F}| = t!(n-t)!$ if and only if $\mathcal{F}$ is a coset
  of a stabilizer of a $t$-subset of $\{1,2,\dots,n\}$.\label{set-ellis}
\end{thm}

The proof of Theorem~\ref{set-ellis} uses similar arguments to the
proof of the Deza-Frankl conjecture in~\cite{ellis2011intersecting}.
In both proofs the result holds for $n$ sufficiently large relative to
$t$ and exact bounds for $n$ are not given. It is conjectured that for
$t$-pointwise intersection the correct lower bound on $n$ is $2t+1$,
while for $t$-setwise intersection it is conjectured
in~\cite{ellis2012setwise} that the result holds for all $n \geq t$.

In this paper we will prove that the conjectured exact lower bound on $n$
hold in both cases for $t=2$. We also give a characterization of the
sets that meet the bound---before we can describe this
characterization, we need to define some terms.. 

The {\it regular module} of $\sym(n)$ is the complex vector space with
basis $\sym(n)$---the elements of this module can be thought of as
vectors of length $|\sym(n)|$. For example, the \textsl{characteristic
  vector} for a subset $S \subset \sym(n)$ is the length-$|\sym(n)|$
vector with the $g$-entry equal to 1 if $g \in S$ and 0 otherwise is a
vector in this module. This module can be identified with the vector
space $\C[\sym(n)]$ which has the structure of a left
$\C[\sym(n)]$-module by left multiplication.  So $\C[\sym(n)]$ can
also be identified with a subalgebra of the $|\sym(n)| \times
|\sym(n)|$-matrices.
 
It is well-known that each irreducible character of $\sym(n)$
corresponds to a partition $\lambda \vdash n$ (denoted by
$\chi^\lambda$) For each such irreducible character, let $E_\lambda$
be the $|\sym(n)| \times |\sym(n)|$-matrix with the $(g,h)$-entry
equal to $\chi^\lambda(hg^{-1})$. We call the image $E_\lambda$
(considered as a linear operator on $\C[\sym(n)]$) the
$\lambda$-module and denote if by $V_\lambda$.

In~\cite{godsil2009new} it is shown that the characteristic vector of
any maximum intersecting set in $\sym(n)$ is a vector in $V_{[n]} \oplus V_{[n-1,1]}$.
In general, it $G$ is any 2-transitive group, then the permutation module is the
sum of two irredicuble representations; the trivial representation and
one denoted by $\phi$. In~\cite{MeagherSin} it is shown that for any
2-transitive group and the characteristic vector of any maximum
intersecting set of permutations $S$ in $G$ lies in the sum of the
trivial and the $\phi$-module. This is call the \textsl{EKR-module
  property}. 

In this paper, we will show that the characteristic vectors of the maximum
2-setwise and 2-pointwise intersecting sets in $\sym(n)$ each lie in a
specific module. We will use spectral methods and the representation
theory of the symmetric group in our proof for Theorem~\ref{main-thm}
and Theorem~\ref{main-pointwise}.

We state our first theorem on the $2$-setwise
action as follows.

\begin{thm}
	Let $n\geq 2$. If $\mathcal{F}$ is a family of $2$-setwise
        intersecting permutations, then $|\mathcal{F}| \leq
        2(n-2)!$. Moreover, for $n\geq 4$, if $S$ is a maximum $2$-setwise
        intersecting family, then its charactersitic vector $\chi_S$
        is in  $V_{[n]} \oplus V_{[n-1,1]} \oplus V_{[n-2,2]}$.\label{main-thm}
\end{thm}

Our second theorem is a similar result for
$2$-pairwise intersecting permutations. We state it as follows.

\begin{thm}
  Let $n\geq 5$. If $\mathcal{F}$ is a $2$-pointwise intersecting family of
  permutations, then $|\mathcal{F}| \leq (n-2)!$. In addition, if $S$
  is a maximum $2$-pointwise intersecting family of $\sym(n)$, then
  $\chi_S$ in $V_{[n]} \oplus V_{[n-1,1]} \oplus V_{[n-2,2]} \oplus V_{[n-2,1,1]}$.\label{main-pointwise}
\end{thm}

We organize this paper as follows. In Section~\ref{prem} and Section~\ref{representations}, we recall
some basics on the symmetric group $\sym(n)$ and combinatorial objects
called weighted adjacency matrix.  Section~\ref{proof-thm-setwise} is
devoted to the proof of Theorem~\ref{main-thm}. In
Section~\ref{point-proof}, we give the proof for
Theorem~\ref{main-pointwise}.
 
\section{Background}\label{prem}

\subsection{Derangement graphs}

One of the techniques to prove the EKR property for a group is to use
the \emph{derangement graph.} The set of derangements of a permutation
group $G\leq \sym(n)$ is the set of all permutations of $G$ without
fixed points. We denote by $\Der(G)$ the set of all derangements of
the permutation group $G$. The derangement graph $\Gamma_G$ is the
undirected graph with vertex set $G$, where two permutations $g,h\in
G$ are adjacent if and only if $hg^{-1}\in \Der(G)$. For the case
where $G$ is symmetric group we denote $\Gamma_{sym(n)}$ by
$\Gamma_n$.

If $\mathcal{F}$ is an intersecting family of $G$, then in $\Gamma_G$,
the elements of $\mathcal{F}$ form an \emph{independent set} or
\emph{coclique}. Therefore, a transitive group has the EKR property if and only
if the size of a maximum coclique is at most $\frac{|G|}{n}$.

Given a graph $X$, we denote by $\alpha(X)$ and $\omega(X)$
respectively the maximum size of a coclique and maximum size of a
clique of $X$. The following result is well-known and a proof can be
found in~\cite[Section 2.1]{godsil2016erdos}.

\begin{lem}[Clique-coclique bound]
	Let $X$ be a vertex-transitive graph on $n$ vertices. Then
	\begin{align*}
		\alpha(X) \omega(X) \leq n.
	\end{align*}\label{clique-coclique}
\end{lem}

\begin{cor}
	The symmetric group $\sym(n)$ has the EKR property.
\end{cor}
\begin{proof}
  First we observe that $\omega(\Gamma_n)$ is at most the degree,
  which is $n$ in this case. Moreover, the rows of a Latin square of
  order $n$ correspond to a set of permutations that forms a clique of
  size $n$ in $\Gamma_n$. Using Lemma~\ref{clique-coclique}, we have
  \begin{align*} 
   \alpha(\Gamma_n) \leq \frac{n!}{n} = (n-1)!.
  \end{align*} 
Hence, $\sym(n)$ has the EKR property. 
\end{proof}

This approach does not work in general for 2-setwise or 2-pointwise
intersection. To see this, consider $\sym(n)$ with 2-pointwise
intersection; a clique in the corresponding derangement graph would be
a sharply 2-transitive subset of $\sym(n)$. Since such a set does not
exist for all $n$, the clique-coclique bound will not always holds
with equality. Rather we will use the ratio bound, as was used
in~\cite{ellis2011intersecting,
  ellis2012setwise,godsil2015algebraic}. The ratio bond is also known
as the Delsarte bound or Hoffman bound, we refer the reader
to~\cite{godsil2015algebraic} for a proof. We use $\mathbf{1}$ to
denote the all ones vector and $J$ for the all ones matrix (the sizes
will be clear from context).
\begin{lem}[Ratio bound]
  Let $X$ be a $d$-regular graph having $n$ vertices. Let
  $\lambda_{\min}$ be the minimum eigenvalue of the adjacency matrix
  of $X$. Then we have
	\begin{align*}
		\alpha(X) \leq \frac{n}{1-\frac{d}{\lambda_{\min}}}.
	\end{align*} \label{ratio-bound}
	Moreover, if $S$ is an independent set of size $\frac{n}{1-\frac{d}{\lambda_{\min}}}$ with characteristic vector $v_S$, then 
	\begin{align*}
		v_S-\frac{|S|}{n}\mathbf{1}
	\end{align*}
	is a $\lambda_{\min}$-eigenvector.
\end{lem}

We will actually use a generalization of the ratio bound which we
state and prove once we give a definition. For a graph $X$ on $n$
vertices, a real symmetric $n \times n$ real matrix $A = (a_{i,j})$ with
constant row and column sum is a \emph{weighted adjacency matrix} (or
a \emph{pseudo-adjacency matrix}) for $X$ if $a_{i,j} = 0$ whenever
$i\not\sim_X j$ (note that $a_{i,j}$ could be 0 for adjacent vertices
$i$ and $j$).

\begin{thm} [Weighted Ratio Bound]~\label{thm:wtRatio} Let $X$ be a
  connected graph.  Let $A$ be a weighted adjacency matrix for $X$ with
  constant row and column sum $d$.

 If the least eigenvalue of $A$ is $\tau$, then
\[
  \alpha(X) \le\frac{|V(X)|}{1-\frac{d}{\tau}}.
\]
Further, if equality holds for some coclique $S$ with characteristic vector
$v_S$, then
\[
v_S -\frac{|S|}{|V(X)|}\one
\]
is an eigenvector with eigenvalue $\tau$. \qed
\end{thm}
\proof
Set $v=|V(X)|$ and $s=|S|$ and denote $A(X)$ by $A$.  Let $\{\one,
w_2,\dots,w_n \}$ be an orthonormal basis of real eigenvectors for $A$
with $Aw_i = \lambda_i$.

Define
\[
	M = A -\tau I -\frac{d-\tau}{v}J.
\]
Then $M \one = 0$, and for $w_i$
\[
Mw_i = (\lambda_i - \tau) w_i
\] 
since $w_i$ is orthogonal to the all ones vector.
Thus all the eigenvalues of $M$ are non-negative, and $M$ is positive
semidefinite. Hence for any vector $x$ 
\begin{equation}\label{ineq:psdtau}
	0 \le x^TMx = x^TAx -\tau x^Tx -\frac{d-\tau}{v}x^TJx.
\end{equation}
Let $x$ be the characteristic vector of a coclique of size $s$, then
$x^TAx=0$ and Equation(~\ref{ineq:psdtau}) simplifies to
\[
0 \le x^TAx -\tau s -\frac{d-\tau}{v}s^2. 
\]
Hence
\[
\tau s \le \frac{d-\tau}{v}s^2
\]
and the inequality follows. If equality holds, then $x^TMx=0$; since $M$ is positive
semidefinite, this implies that $Mx=0$.  Therefore
\begin{equation}
	(A-\tau I)x =\frac{d-\tau}{v} Jx = \frac{s}{v} (d-\tau)\one = (A-\tau I)  \frac{s}{v} \one 
\end{equation}
setting $x = v_S-\frac{s}{v}\one$ implies the second claim.\qed

\subsection{Eigenvalues of normal Cayley graphs}

Let $G$ be a group and let $C$ be an inverse closed subset of
$G\setminus\{1\}$. A Cayley graph $\cay(G,C)$ is graph whose vertex
set is the set $G$, and two group elements $g$ and $h$ are adjacent if
$hg^{-1}\in C$. A Cayley graph is called \emph{normal} if the set $C$
is closed under conjugation. That is, for any $g\in G$, we have
$gCg^{-1} \subseteq C$.  The derangement graph for a group is a Cayley
graph, in particular, $\Gamma_G=\cay(G,\Der(G))$.  Since $\Der(G)$ is
closed under conjugation, the derangement graph is a normal Cayley
graph.  In this subsection, we review some properties of Cayley graphs
and give a formula for the eigenvalues of normal Cayley graphs.

The eigenvalues of a Cayley graph
$\cay(G,C)$ can be obtained from the irreducible
characters of the group $G$. We present this result in the following
lemma whcih is usually attributed to Babai~\cite{babai1979spectra}, or Diaconis and
Shahshahani \cite{MR626813}; a proof may be found in~\cite[Section
11.12]{godsil2009new}.

\begin{lem}[Babai~\cite{babai1979spectra}]
	Let $\operatorname{Irr}(G)$ be the set of all irreducible
        representations on a group $G$. If $X = \cay(G,C)$ is a
        normal Cayley graph, then eigenvalues of the adjacency matrix
        of $X$ are given by
	\begin{align*}
	\xi_{\chi} &= \frac{1}{\chi(\id)} \sum_{g\in C} \chi(g),
	\end{align*}\label{eig-char}

        for $\chi \in \operatorname{Irr}(G)$. The multiplicity of $\xi$ is given by $\sum m_\chi^2$
        where $m_\chi$ is the dimension of $\chi$, and the sum is taken
        over all irreducible representations $\chi$ with $\xi_\chi
        = \xi$.
\end{lem}

We will use a weighted adjacency matrix that is formed by taking a
linear combination of matrices from the conjugacy class association
scheme on $\sym(n)$. Next we show how to
calculate the eigenvalues of such a matrix.

For $\rho$ a partition of $n$, let $C_\rho$ represent the
conjugacy classes of $\sym(n)$ with shape $\rho$, and define the
$n!  \times n!$ matrix
\begin{align*}
A_\rho[g,h] &= \begin{cases}
		1 & \mbox{ if }hg^{-1}\in C_\rho,\\
		0 & \mbox{ otherwise.}
               \end{cases} 
\end{align*}
The set of matrices $\mathcal{A} = \{A_\rho\}_{\rho\vdash n}$ form an
association scheme called the \textsl{conjugacy class scheme} on
$\sym(n)$--- in particular $A_\rho$ is the matrix in this association
scheme corresponding to the conjugacy class of permutations with shape
$\rho$ (see~\cite[Section 3.3]{godsil2016erdos} for details about this
association scheme). For a conjugacy class of derangements
$C_\rho$, the matrix $A_\rho$ is the adjacency matrix of
$\cay(\sym(n),C_\rho)$, and by Lemma~\ref{eig-char} the eigenvalues
of $A_\rho$ are given by
$\frac{|C_\rho|}{\chi(\id)}\chi(g_\rho)$ where $g_\rho \in
C_\rho$ and $\chi$ is an irreducible representation of
$\sym(n)$~\cite{godsil2015algebraic}.  

We will consider weighted adjacency matrices of the derangement
graphs of $\sym(n)$ with the 2-setwise and the 2-pointwise
actions. The adjacency matrices we consider have a constant
weight on each conjugacy class; hence we only consider matrices in the
conjugacy class association scheme for $\sym(n)$. The next result
gives a formula for the eigenvalues of such a weighted adjacency matrix.

\begin{lem}\label{lem:weightedEvalues}
  Let $\{ C_\rho \,|\, \rho \vdash n\}$ be the conjugacy classes of
  $\sym(n)$, let $A_\rho$ be the matrix of the conjugacy classes scheme
  of $\sym(n)$ that corresponds to $C_\rho$. Let $A = \sum_{\rho
    \vdash n}  \omega_\rho A_\rho$, where $(\omega_\rho)_{\rho\vdash n} \subseteq
  \mathbb{R}$. Then the eigenvalues of $A$ are
\[
\xi_\chi = \sum_{\rho \vdash n}\omega_\rho \frac{|C_\rho|}{\chi(\id)}\chi(g_\rho)
\]
where $g_\rho \in C_\rho$ and $\chi$ is an irreducible representation of $\sym(n)$.
\end{lem}

If the conjugacy class $C_\rho$ is a conjugacy class of
derangements under the action, then it is called a \emph{derangement
  class}.  For the setwise action on 2-sets, a derangement class is
any class $C_\rho$ in which $\rho$ has no parts of size $2$. For
the pointwise action, a derangement class is any class with a
corresponding partition that contains at most one part of size $1$.

\section{Representations of $\sym(n)$}\label{representations}

In this section we state the results on the representation theory of
the symmetric group that we need. We do not prove these results, but
we refer the reader to Sagan \cite{sagan2001symmetric} or another such book.

The degree of a representation is given by the value of the character
on the identity. For the symmetric group, the degree of an irreducible
representation can be computed via the well-known hook length formula.

If $\lambda \vdash n$, then a pair $(i,j)$ is a \emph{node} of the
Young diagram of $\lambda$ if the path downward $j$ cells from the top
leftmost cell and then rightward $i$ cells ends on a cell of
$\lambda$. By $(i,j)\in \lambda$, we mean $(i,j)$ is a node of
$\lambda$.  For a node $(i,j)$ define
\[
H_{i,j} = \left\{ (u,j) \in \lambda \mid i\leq u \right\} \cup
\left\{ (i,v)\in \lambda \mid j \leq v \right\}
\]
and $h_{i,j} = \left| H_{i,j} \right|$.

\begin{thm}[Hook Length Formula~\cite{frame1954hook}] 
  Let $\lambda \vdash n$. The degree of the character $\chi^\lambda$
  (this is the character corresponding to $\lambda$) is
\begin{align*}
\chi^\lambda(\id) = \frac{n!}{\prod_{(i,j) \in \lambda} h_{i,j}}
\end{align*} 
where the product is taken over all nodes in the Young's diagram of $\lambda$.
\end{thm}

We will need to consider the low-dimensional representations of
$\sym(n)$ seperately from the ones with higher dimension. We will not
give a proof to this since a similar result is proved in~\cite{godsil2016erdos}.

\begin{lem}
  Let $n\geq 13$. If $\chi^\lambda$ is an irreducible representation of
  $\sym(n)$ of dimension less than $2\binom{n+1}{2}$, then $\lambda$ is
  one of the following : $[n]$,$[1^n]$, $[n-1,1]$, $[2,1^{n-2}]$,
  $[n-2,2]$, $[2^2,1^{n-4}]$, $[n-2,1^2]$ or
  $[3,1^{n-3}]$.\label{low-dim}
\end{lem}

Next we state the well-known recursive formula for calculating the
value of a character of $\sym(n)$ on a conjugacy class. Let
$\chi_\rho^\lambda$ be the value of the character 
$\chi^\lambda$ on an element of the conjugacy class $C_\rho$.
A composition is an unordered partition. We will use square brackets
for the partitions corresponding to irreducible representations, and
round bracket for partitions corresponding to the conjugacy classes of $\sym(n)$.

\begin{lem}[Murnaghan-Nakayama Rule]
  If $\lambda \vdash n$ and $\rho$ is a composition of $n$, with $\rho
  = (\rho_1,\rho_2,\ldots,\rho_k)$, then
\begin{align*}
	\chi^\lambda_{\rho} &= \longsum_{\xi \in {\sf RH_{\rho_1}}(\lambda)} (-1)^{\ell\ell(\xi)} \chi_{\rho \backslash \rho_1}^{\lambda\setminus \xi},
\end{align*}
where ${\sf RH}_{\rho_1}(\lambda)$ is the set of all rim hooks with
$\rho_1$ cells of $\lambda$, and $\ell\ell(\xi)$ is the number of rows
the rim hook spans minus one.\label{MurnNakRule}
\end{lem}

We will build weighted adjacency matrices for the derangement graphs
for $\sym(n)$ with the setwise and the pointwise action.  We will then
calculate the eigenvalues of these matrices and prove that the ratio
bound holds with equality. To calculate the eigenvalues, we will need
to determine the value of the irreducible representations on specific
conjugacy classes. In our weighting, we weight many of the conjugacy
classes to be 0, so we only need to consider the values of the
irreducible representations that have a non-zero weight. The next
result gives the values of the irreducible representation on these
specific conjugacy classes.

\begin{lem}
  The irreducible characters of $\sym(n)$ that do not vanish on the
  conjugacy classes $C_{(n)}$, $C_{(n-1,1)}$, $C_{(n-3,3)}$,
  $C_{(n-4,3,1)}$, $C_{(n-2,2)}$ and
  $C_{(n-3,2,1)}$ are given in Table~\ref{char-val-pointwise}.
\end{lem}

We end this section with considering two representations of the
symmetric group.  The Young's subgroup $\sym(n-2) \times \sym(2)
=\sym([n-2,2]) $ is the setwise stabilizer of a $2$-set. Thus the
setwise action of $\sym(n)$ sets of size $2$ has representation
\[
\Ind^{\sym(n)}(1_{\sym([n-2,2])}) =\chi^{[n]} + \chi^{[n-1,1]} + \chi^{[n-2,2]}.
\]
Similarly, the Young's subgroup $\sym(n-2) \times \sym(1) \times \sym(1) =\sym([n-2,1,1]) $ is the
pointwise stabilizer of a $2$-set. Thus the pointwise action of $\sym(n)$ on sets of size $2$ has
representation  
\[
\Ind^{\sym(n)}(1_{\sym([n-2,1,1])}) =\chi^{[n]} + 2\chi^{[n-1,1]} + \chi^{[n-2,2]}+ \chi^{[n-2,1,1]}.
\]
Both of these decompositions follow from calculating the Kostka numbers.

We will construct a weighted adjacency matrix for the derangement
graph for each action and the ratio bound will hold with equality for
this matrix.  In the weighted adjacency matrix that we construct, the
non-trivial representations in the decomposition of each of the
representation of the permutation action will be exactly the
representations that achieve the minimal eigenvalue. This fact will
show that the characteristic vectors of any maximum coclique will lie
in a specific $\sym(n)$-module.

\section{Proof of Theorem~\ref{main-thm}}\label{proof-thm-setwise}

In this section we will give the weighted adjacency
matrix for the derangement graph of $\sym(n)$ with the 2-setwise
action; this matrix has the form
\begin{align}
A = \omega_1A_{(n)} + \omega_2A_{(n-1,1)} + \omega_3A_{(n-3,3)} + \omega_4 A_{(n-4,3,1)}\label{w-mat}
\end{align}
for some positive numbers $\omega_1$ $\omega_2$, $\omega_3$ and
$\omega_4$ which are to be determined. We are choosing only four of
the derangement classes in the association scheme (namely the
conjugacy classes with cycle types $(n)$, $(n-1,1)$, $(n-3,3)$ and
$(n-4,3,1)$) to have a non-zero weighting. The sizes of the conjugacy
classes with cycle types $(n),\ (n-1,1),\ (n-3,3)$ and
$(n-4,3,1)$ are, respectively,
\[
\alpha =(n-1)!, \quad \beta = n (n-2)!, 
\quad  \gamma = 2\binom{n}{3}(n-4)!, \quad \delta = 8\binom{n}{4}(n-5)!
\]

We choose our weighting so that the following three conditions hold:
\begin{enumerate} 
\item the trivial representation gives the
eigenvalue $\binom{n}{2}-1$;
\item the nontrivial irreducible characters
that are in the decomposition of the permutation action (namely $\chi^{[n-1,1]}$
and $\chi^{[n-2,2]}$)  have
eigenvalue $-1$; and 
\item all other representations give eigenvalues strictly between
  $\binom{n}{2}-1$ and $-1$.
\end{enumerate}

It is straight-forward to calculate the eigenvalues of the adjacency matrices for
the four conjugacy classes we have chosen corresponding to the irreducible
representations in the decomposition of the permutation action. This
values are in the following table.

\begin{table}[h]
	\begin{tabular}{|c|c|c|c|c|} \hline
		& $A_{(n)}$  & $A_{(n-1,1)}$ & $A_{(n-3,3)}$ & $A_{(n-4,3,1)}$ \\
		Representation  & & & &\\ \hline
		$\chi^{[n]}$      &  $\alpha $ & $\beta$ & $\gamma$  & $\delta$   \\ \hline
		$\chi^{[n-1,1]}$ & $-\frac{\alpha}{n-1}$ & $0$  & $-\frac{\gamma}{n-1}$ & $0$ \\ \hline
		$\chi^{[n-2,2]}$ &$0$ &  $-\frac{2\beta}{n(n-3)}$ & $0$ & $-\frac{2\delta}{n(n-3)}$  \\ \hline
	\end{tabular}
	\caption{Eigenvalues of $A_{(n)}$, $A_{(n-1,1)},\ A_{(n-3,3)}$ and $A_{(n-4,3,1)}$ afforded by $\chi^{[n]}$, $\chi^{[n-1,1]}$ and $\chi^{[n-2,2]}$.}
\end{table}

By Lemma~\ref{lem:weightedEvalues}, to find weightings that satisfy
conditions (1) and (2) above we
need to solve the following three linear equations.

\begin{align}
\begin{aligned}
\omega_1 \alpha + \omega_2\beta + \omega_3 \gamma + \omega_4 \delta &= \binom{n}{2}-1\\
 -\omega_1 \alpha - \omega_3 \gamma &= -(n-1) \\
- \omega_2\beta -\omega_4 \delta &= -\frac{n(n-3)}{2}.
\end{aligned}\label{weights-eq-setwise}
\end{align}

This linear system has infinitely many solutions with two free
variables. The general solution to \eqref{weights-eq-setwise} has the form
\begin{align*}
	\omega_1(s,t) &=  \frac{1}{\alpha} \left( -\gamma s + (n-1) \right)\\
	\omega_2(s,t) &= \frac{1}{\beta} \left ( -\delta t +\frac{n(n-3)}{2} \right)\\
	\omega_3(s,t) &= s\\
	\omega_4(s,t) &= t.
\end{align*}

Thus, expressed in these two free variables, any eigenvalue of $A$ is of the form
\begin{align}
	\xi_{\chi} &= \frac{1}{\chi(\id)} \left( \left((n-1) -\gamma s\right) \chi_{(n)} + \left(\frac{n(n-3)}{2} - \delta t\right)\chi_{(n-1,1)} + \gamma s \chi_{(n-3,3)} + \delta t \chi_{(n-4,3,1)} \right)\label{eig-set}
\end{align}
for $\chi \in \operatorname{Irr}(\sym(n))$. In other words,
eigenvalues of $A$ are functions of the parameters $s$ and $t$.  For
the remaining part of the proof, we shall write the eigenvalues for an
irreducible character $\chi$ in function of the parameters $t$ and $s$
(that is, $\xi_{\chi}(s, t)$). We will choose values of $s$ and
$t$ so that all the eigenvalues of $A$ satisfy all three conditions
listed above.

To do this, we will define a polytope $\mathcal{P}$ and show for any
values of $(s,t)$ in $\mathcal{P}$, the weightings $\omega_i(s,t)$
give a matrix $A$ that satisfies the three conditions above.
We then apply Theorem~\ref{thm:wtRatio} to this $A$, which shows that
Theorem~\ref{main-thm} holds.

Define $(\mathcal{P})$ to be the polytope that is the intersection of the
following halfspaces of $\mathbb{R}^2$
\begin{align*}
	(\mathcal{P}) \quad \left\{ \qquad
	\begin{aligned}
		&2\gamma x - 2\delta y + \binom{n-1}{2} - (n-1) +2 <0\\
		&2\gamma x - 2\delta y + \binom{n-1}{2} - (n-1) >0\\
		&0< \gamma x < n-1,\\ 
               &0< \delta y < \frac{n(n-3)}{2}.
	\end{aligned}
	\right. 
\end{align*}
The polytope ($\mathcal{P}$) is non-empty since the first two
equations are those of parallel lines and they intersect the rectangle
formed by the last two equations.  Note that the final two equations
imply the following result.

\begin{lem}
  For any $(s,t) \in \mathcal{P}$, the weightings $\omega_1(s,t)$,
  $\omega_2(s,t)$, $\omega_3(s,t)$ and $\omega_4(s,t)$ are positive.
\end{lem}

Next we will determine the eigenvalues of $A$ so that
Theorem~\ref{thm:wtRatio} can be applied.

\begin{lem}\label{lem:poscoeffient1}
Let $n\geq 11$. For any $(s,t) \in \mathcal{P}$, the eigenvalues of the matrix 
\[
A = \omega_1(s,t)A_{(n)} + \omega_2(s,t)A_{(n-1,1)} +
        \omega_3(s,t)A_{(n-3,3)} + \omega_4(s,t) A_{(n-4,3,1)}
\]
are in $[-1,\binom{n}{2}-1]$. Moreover, the eigenvalues for
$\chi^{[n-1,1]}$ and $\chi^{[n-2,2]}$ are the only ones equal to $-1$.
\end{lem}

\begin{proof}
	Let $\chi\in \operatorname{Irr}(\sym(n))$ be such that
        $\chi(\id) >\binom{n}{2}$. Using the triangle inequality and Lemma~\ref{lem:poscoeffient1} on
        \eqref{eig-set}
	\begin{align*}
		\left|\xi_{\chi}\right| 
                   &\leq \frac{1}{\binom{n}{2}}\left(  \left((n-1) -\gamma s\right) |\chi_{(n)}| 
                + \left(\frac{n(n-3)}{2} - \delta
                  t\right)|\chi_{(n-1,1)}| + \gamma s |\chi_{(n-3,3)}| 
                + \delta t |\chi_{(n-4,3,1)}| \right)\\
		&\leq \frac{1}{\binom{n}{2}} \left(  \left((n-1)
                    -\gamma s\right)  + \left(\frac{n(n-3)}{2} -
                    \delta t \right) + \gamma s  + \delta t \right)\\
		&=  \frac{1}{\binom{n}{2}} \left(  (n-1)  +\frac{n(n-3)}{2} \right)\\
		&= \frac{1}{\binom{n}{2}} \left(    \binom{n}{2}-1 \right)\\
             &<1.
	\end{align*}
	
	Using Lemma~\ref{low-dim}, the eigenvalues for irreducible characters of degree less than $\binom{n}{2}$ are
	\begin{align*}
		\xi_{\chi^{[n]}} &= \binom{n}{2}-1,\\
		\xi_{\chi^{[n-1,1]}} &= \xi_{\chi^{[n-2,2]}} = -1\\
		\xi_{\chi^{[n-2,1^2]}} &= \frac{2}{n-2}\\
		\xi_{\chi^{[3,1^{n-3}]}} &= \frac{2(-1)^{n-1}}{n-2} +
                (-1)^{n}\frac{2\gamma s }{\binom{n-1}{2}} \\
		\xi_{\chi^{[2^2,1^{n-4}]}} &= (-1)^{n-1} + (-1)^n\frac{4 \delta t}{n(n-3)}  \\
		\xi_{\chi^{[2,1^{n-2}]}} &= (-1)^{n} + (-1)^{n-1}\frac{2\gamma s}{n-1}\\
		\xi_{\chi^{[1^n]}} &= (-1)^n \left( \binom{n}{2} - 2(n-1) -1 + 2\gamma s - 2 \delta t \right)\\
	\end{align*}

        One can immediately see that, with the exception of
        $\xi_{\chi^{[n]}}$, these eigenvalues are all strictly less
        than $\binom{n}{2}-1$.

        So we need to show that the eigenvalues
        $\xi_{\chi^{[3,1^{n-3}]}}$, $\xi_{\chi^{[2^2,1^{n-4}]}}$,
        $\xi_{\chi^{[2,1^{n-2}]}}$ and $\xi_{\chi^{[1^n]}}$ are
        strictly greater than $-1$ whenever $(s,t) \in
        \mathcal{P}$.

When $n$ is even, the eigenvalues are larger than $-1$ if
\begin{enumerate}
\item $2\gamma s - 2 \delta t + \binom{n-1}{2} - (n-1) >0$,
\item $\gamma s <n-1$,
\item $0 < t$
\end{enumerate}

When $n$ is odd, the eigenvalues are larger than $-1$ if 
\begin{enumerate}
\item $2\gamma s - 2\delta t + \binom{n-1}{2} - (n-1) +2 <0$, 
\item $0 < s$,
\item $\delta t<\frac{n(n-3)}{2}$.
\end{enumerate}
	
Combining the cases when $n$ is even and odd, we obtain exactly the
equations of $(\mathcal{P})$. Since $(\mathcal{P})$ is not empty, all
the eigenvalues are greater than $-1$ for any $(s,t) \in
\mathcal{P}$. This completes the proof.
\end{proof}

\begin{proof}[Proof of Theorem~\ref{main-thm} ]
	For $2\leq n\leq 10$, we use Sagemath~\cite{sagemath} to prove the result. For $n\geq 11$, we use Theorem~\ref{thm:wtRatio}, we have
	\begin{align*}
	\alpha(\Gamma_{\sym(n)})\leq \alpha(X)\leq \frac{n!}{1 - \frac{\binom{n}{2}-1}{-1}} = 2(n-2)!.
	\end{align*}
\end{proof}

This result shows that equality holds in Theorem~\ref{thm:wtRatio};
thus, if $S$ is a maximum 2-setwise intersecting set, then
$v_S - \left( \binom{n}{2} \right)^{-1} \one$ is a $-1$-eigenvector for $A$. Since the only irreducible
representations that give the least eigenvalue $-1$ are $\chi^{[n-1,1]}$
and $\chi^{[n-2,2]}$, we have the following corollary.

\begin{cor}
	For $n\geq 4$, any characteristic vector of a maximum 2-setwise
	intersecting set of permutations in $\sym(n)$ is in
	the module $V_{[n]}\oplus V_{[n-1,1]} \oplus V_{[n-2,2]}$.
\end{cor}

\section{Proof of Theorem~\ref{main-pointwise}}\label{point-proof}

In this section we prove Theorem~\ref{main-pointwise} by constructing
a weighted adjacency matrix for the derangement graph of $\sym(n)$
with the $2$-pointwise action. This weighted adjacency matrix will be
a linear combination of the adjacency matrices in the conjugacy
classes scheme for $\sym(n)$ corresponding to the conjugacy classes
with cycle types $(n)$, $(n-1,1)$, $(n-2,2)$, $(n-3,3)$,$(n-3,2,1)$
and $(n-4,3,1)$.  In particular, we set
\begin{align}
A &= \omega_1 A_{(n)} + \omega_2 A_{(n-1,1)}+\omega_3 A_{(n-2,2)} + \omega_4 A_{(n-3,3)} + \omega_5 A_{(n-3,2,1)} + \omega_6 A_{(n-4,3,1)}.\label{w-mat-pointwise}
\end{align}

Similar to the proof of Theorem~\ref{main-thm}, we will find
$(\omega_i)$ for $i=1,\ldots,6 $ so that the following three conditions hold:
\begin{enumerate} 
\item the trivial representation gives the
eigenvalue $2\binom{n}{2}-1$;
\item the nontrivial irreducible characters
that are in the decomposition of the 2-pointwise permutation action (namely
$\chi^{[n-1,1]}$, $\chi^{[n-2,2]}$ and $\chi^{[n-2,1,1]}$)  have
eigenvalue $-1$; and 
\item all other representations give eigenvalues strictly between
  $2\binom{n}{2}-1$ and $-1$.
\end{enumerate}

Define
\begin{align*}
\alpha = (n-1)!, &\quad &
\beta = n(n-2)!, &\quad& 
\gamma = \binom{n}{2}(n-3)!,  \\
\delta = 2\binom{n}{3}(n-4)!, &\quad &
\mu = 3\binom{n}{3}(n-4)!, &\quad & 
\nu = 8 \binom{n}{4}(n-5)!. 
\end{align*}

These numbers are respectively the sizes of the conjugacy classes with
cycle type $(n)$, $(n-1,1)$, $(n-2,2)$, $(n-3,3)$, $(n-3,2,1)$ and $(n-4,3,1)$. 
The following table gives the eigenvalues values of the matrices
$A_{(n)}$, $A_{(n-1,1)}$, $A_{(n-2,2)}$, $A_{(n-3,3)}$,
$A_{(n-3,2,1)}$ and $A_{(n-4,3,1)}$ in the conjugacy class association
scheme corresponding to these irreducible characters.

\renewcommand\arraystretch{1.5}

\begin{table}[H]
	\centering
	\begin{tabular}{ |p{3.5cm}|c|c|c|c|c|c|}
		\hline
		 \centering &\centering $A_{(n)}$ & $A_{(n-1,1)}$ & \centering $A_{(n-2,2)}$ & $A_{(n-3,3)}$ & $\ \  A_{(n-3,2,1)}$ & $A_{(n-4,3,1)}$ \\
		\centering Representation & & & &  & &\\
		\hline 
		\centering $\chi^{[n]}$ &  \centering $\alpha$ &\centering $\beta$ &\centering $\gamma$ & $\delta$ & $ \mu$   & $\nu$\\
		\hline
		\centering $\chi^{[n-1,1]}$&  \centering $-\frac{\alpha}{n-1}$ &\centering $0$&\centering $-\frac{\gamma}{n-1}$ & $\frac{-\delta}{n-1}$  & $ 0$ & $0$\\
		\hline
		\centering $\chi^{[n-2,2]}$&  \centering $0$ &\centering $-\frac{2\beta}{n(n-3)}$&\centering $\frac{2\gamma}{n(n-3)}$ & $0$ &  $ 0$ & $-\frac{2\nu}{n(n-3)}$\\
		\hline
		\centering $\chi^{[n-2,1,1]}$&  \centering $\frac{\alpha}{\binom{n-1}{2}}$ &\centering $0$&\centering $0$ & $\frac{\delta}{\binom{n-1}{2}}$  &  $  -\frac{\mu}{\binom{n-1}{2}}$ & $0$\\
		\hline 
	\end{tabular}
	\caption{Eigenvalues afforded by $\chi^{[n]}$, $\chi^{[n-1,1]}$, $\chi^{[n-2,2]}$ and $\chi^{[n-2,1,1]}$ on $A_{(n)}$, $A_{(n-1,1)}$, $A_{(n-2,2)},$ $A_{(n-3,3)}$, $A_{(n-3,2,1)}$ and $A_{(n-4,3,1)}$.}
\end{table}
\renewcommand\arraystretch{1}

Using Lemma~\ref{lem:weightedEvalues}, it is straight-forward to
calculate the eigenvalues of $A$ afforded by $\chi^{[n]}$,
$\chi^{[n-1,1]}$, $\chi^{[n-2,2]}$ and $\chi^{[n-2,1,1]}$ as functions
of $\omega_i$ where $i =1,\dots, 6$.  Thus, in order to satisfy
conditions (1) and (2) above, the $\omega_i$ must satisfy the
following system of linear equations.
\begin{align}
\begin{aligned}
\alpha \omega_1 + \beta \omega_2 + \gamma \omega_3 + \delta \omega_4 + \mu \omega_5 + \nu \omega_6 &= n(n-1)-1\\
-\alpha \omega_1 - \gamma \omega_3 - \delta \omega_4 &= -(n-1)\\
-\beta \omega_2 +\gamma \omega_3 - \nu \omega_6 &= - \frac{n(n-3)}{2}\\
\alpha \omega_1 + \delta \omega_4  - \mu \omega_5 &= - \binom{n-1}{2}.
\end{aligned}
\label{eq_conj_pointw}
\end{align}

The system~\eqref{eq_conj_pointw} has infinitely many solutions with
three free variables. The following is a general solution to~\eqref{eq_conj_pointw}.

\begin{align}
		\begin{aligned}
                  \omega_1(r,s,t) &= \frac{1}{\alpha} \left( \binom{n}{2}-1 -\beta r -\nu t - \delta s \right),\\
                  \omega_2(r,s,t) &= r,\\
                  \omega_3(r,s,t) &= \frac{1}{\gamma} \left( 1- \binom{n-1}{2} + \beta r + \nu t \right),\\
                  \omega_4(r,s,t) &= s,\\
                  \omega_5(r,s,t) &= \frac{1}{\mu} \left( \binom{n}{2} + \binom{n-1}{2} - 1 -\beta r  - \nu t \right)\\
                  \omega_6(r,s,t) &= t
		\end{aligned}
	\label{point-weights}
\end{align}
for $r,s,t\in \mathbb{R}$.

The eigenvalue of $A$ corresponding to $\chi$, denoted by
$\xi_{\chi}(r,s,t)$, as a function of $r$, $s$, and $t$ is the following.
\begin{align} 	\label{point-eig}
	\xi_{\chi} (r,s,t) &= \frac{1}{\chi(\id)}  \left( 
                       \left( \binom{n}{2} -1 -\beta r - \delta s - \nu t \right)\chi_{(n)} 
                  + \beta r \chi_{(n-1,1)}  \right. \\ \nonumber
              & \quad + \left( 1 - \binom{n-1}{2} + \beta r + \nu t  \right)\chi_{(n-2,2)} 
                 +\delta s \chi_{(n-3,3)}\\ \nonumber
              & \quad \left. + \left( \binom{n}{2} + \binom{n-1}{2} -1 - \beta r - \nu t \right)\chi_{(n-3,2,1)}
        + \nu t \chi_{(n-4,3,1)}    \right). \nonumber
\end{align}

In particular, we have
\begin{align*}
	\xi_{\chi^{[1^n]}} &=  (-1)^{n} \left( 4\beta r +2\nu t + 2 \delta s + 3 - 2 \binom{n}{2} - 2\binom{n-1}{2} \right) \\
	\xi_{\chi^{[2,1^{n-2}]}} &= \frac{(-1)^{n}}{n-1} \left( \binom{n}{2} + \binom{n-1}{2} -2 - 2\beta r -2\delta s -2\nu t \right)\\
	\xi_{\chi^{[2^2,1^{n-4}]}} &= (-1)^{n-1} + (-1)^{n}\frac{4\nu  t}{n(n-3)} \\
	\xi_{\chi^{[3,1^{n-3}]}} &= (-1)^{n} + (-1)^{n}\frac{2\delta s}{\binom{n-1}{2}}.
\end{align*}

We will distinguish the cases $n$ even and $n$ odd, for each case we
will pick values for $r$, $s$ and $t$ so that the matix $A$ satisfies
the conditions in Theorem~\ref{thm:wtRatio}.

\subsection{Subcase 1: $n$ even}

For this case we set $s=0$; this removes the conjugacy class
with cycle type $(n-3,3)$ from the weighted adjacency matrix. 
With $s=0$ and $n$ is even, we can calculate the eigenvalues for the
eight irreducible characters of degree less than $2\binom{n+1}{2}$ as
follows.
\begin{align}
\begin{aligned}
	 \xi_{\chi^{[n]}} &= 2\binom{n}{2}-1\\ 
	 \xi_{\chi^{[n-1,1]}} &= \xi_{\chi^{[n-2,2]}} = \xi_{\chi^{[n-2,1^2]}} = -1\\
	 \xi_{\chi^{[3,1^{n-3}]}} &= 1 \\
	 \xi_{\chi^{[2^2,1^{n-4}]}} &= -1 + \frac{4\nu  t}{n(n-3)} \\ 
	 \xi_{\chi^{[2,1^{n-2}]}} &= \frac{1}{n-1} \left( \binom{n}{2} + \binom{n-1}{2} -2 - 2\beta r  -2\nu t \right)\\
	 \xi_{\chi^{[1^n]}} &=    4\beta r +2\nu t + 3 - 2 \binom{n}{2} - 2\binom{n-1}{2}  
\end{aligned}
\label{eq:evenevalues}
\end{align}

Let $(\mathcal{P}')$ be the polytope obtained by the
following equations of halfspaces of $\mathbb{R}^2$
\begin{align*}
(\mathcal{P}') \quad \left\{ \qquad
\begin{aligned}
& 2\beta  x +\nu y - \binom{n}{2}  - \binom{n-1}{2} +2  > 0\\
& \beta x + \nu y + 1 -\binom{n}{2} < 0\\
& 0 <y .
\end{aligned}
\right.
\end{align*}
This polytope is defined so that the eigenvalues $
\xi_{\chi^{[2^2,1^{n-4}]}}$, $\xi_{\chi^{[2,1^{n-2}]}}$ and
$\xi_{\chi^{[1^n]}}$ are all strictly greater than $-1$.  We let the
reader verify that $(\mathcal{P}')$ is a triangle (without the
boundary) with coordinates
\[
\left(\frac{1}{\beta} \left( \binom{n-1}{2}-1 \right),
      \frac1\nu ({n-2})\right), 
\quad
\left( \frac{1}{\beta} \left( \binom{n}{2}-1 \right) , 0 \right),
\quad 
\left( \frac{1}{\beta} \left(  \binom{n-1}{2} -1 +\frac{n-1}{2} \right), 0 \right).
\]
In particular, $(\mathcal{P}')$ is non-empty.

Next, we prove that for any $(r , t) \in
\mathcal{P}'$, the weightings $w_i$ are all non-negative.

\begin{lem}~\label{lem:evenpositive}
	The weighting $\omega_i(r,0,t)$ is non-negative for any
        $i=1,\ldots,6$, and $(r,t) \in \mathcal{P}'$.
\end{lem}
\begin{proof}
	Let $(r,t)\in \mathcal{P}'$. From the equations of
        $\mathcal{P}$, we have 
\[
-\beta x- \nu y > 1 - \binom{n}{2}, \qquad
\beta x+\nu y > 2\binom{n}{2} + 2\binom{n-1}{2} -2 -\beta x.
\] 
Using these relations and the fact that $\beta x\in
[\binom{n-1}{2}-1,\binom{n}{2}-1 ]$, one can derive that the
weightings are indeed positive.
\end{proof}

Next, we prove that the eigenvalues of the weighted adjacency matrix
are also in the correct range whenever $(r,t)\in \mathcal{P}'$.

\begin{lem}
  Let $n\geq 13$ and even. If $(r,t) \in \mathcal{P}'$ the eigenvalues of the weighted
  adjacency matrix $A$ defined in \eqref{w-mat-pointwise} are in the
  interval $[-1,2\binom{n}{2}-1]$. In addition, the only irreducible
  character giving the eigenvalue $2\binom{n}{2} - 1$ is $\chi^{[n]}$; and the
  only irreducible characters giving eigenvalue $-1$ are
  $\chi^{[n-1,1]},\ \chi^{[n-2,2]}$ and
  $\chi^{[n-2,1^2]}$.\label{pointwise-ev}
\end{lem}

\begin{proof}
  First we show that the statement holds for all irreducible
  characters $\chi$ of $\sym(n)$ with
$\chi(\id) >2\binom{n+1}{2}$. On each of the conjugacy classes that we
consider, the value of any irreducible representation is bounded by
$1$. For any $(r,t) \in \mathcal{P}'$ we have the following bound.
	 \begin{align*}
	 	|\xi_{\chi}(r,0,t)| &\leq \frac{1}{2\binom{n+1}{2}}
                \left( \left(\binom{n}{2}-1 -\beta r -\nu t \right)
                  \left|\chi_{(n)}\right| + \beta r
                  \left|\chi_{(n-1,1)}\right| \right.\\ 
                 &\quad +\left. \left(1-\binom{n-1}{2} + \beta r + \nu t \right) \left|\chi_{(n-2,2)}\right|  \right. \\
	 	 &\quad +\left. \left(\binom{n}{2}+\binom{n-1}{2} -1 -\beta r -\nu t \right) \left|\chi_{(n-3,2,1)}\right| +\nu t\left|\chi_{(n-4,3,1)}\right|  \right)\\
	 	&= \frac{2\binom{n}{2}-1}{2\binom{n+1}{2}}<1.
	 \end{align*}

Thus the eigenvalue of $A$ for $\chi$ is bounded in
  absolute value by $1$ and the statement holds for all irreducible
  characters except $[n]$,$[1^n]$, $[n-1,1]$, $[2,1^{n-2}]$,
  $[n-2,2]$, $[2^2,1^{n-4}]$, $[n-2,1^2]$ or $[3,1^{n-3}]$.

  It is straightforward to see that the non-trivial character all give
  eigenvalue less than $2\binom{n}{2}-1$.

Finally, Equations~\ref{eq:evenevalues} and the definition of
$\mathcal{P}'$ show that the statement holds for all other irreducible characters.
\end{proof}


\subsection{Subcase 2: $n$ odd}

If $n$ is odd, we first note that in the expression of the eigenvalue
for $[3,1^{n-3}]$, the value of $s$ must to be negative for the
inequality $\xi_{\chi^{[3,1^{n-3}]}} >-1$ to hold. We will only use
five conjugacy classes, so we drop conjugacy classes with cycle type
$(n-4,3,1)$ by making $t = 0$.  As in the previous case, we will
consider the irreducible representations wih degree less than
$2\binom{n+1}{2}$ and greater than $2\binom{n+1}{2}$ separately.  The
eigenvalues belonging to representations with degree less than
$2\binom{n+1}{2}$ are
	\begin{align*}
	\xi_{\chi^{[n]}} &= 2\binom{n}{2}-1. \\
	\xi_{\chi^{[n-1,1]}} &= \xi_{\chi^{[n-2,2]}} = \xi_{\chi^{[n-2,1^2]}} = -1\\
	\xi_{\chi^{[3,1^{n-3}]}} &= -1 -\frac{2\delta s}{\binom{n-1}{2}}\\
	\xi_{\chi^{[2^2,1^{n-4}]}} &= 1   \\
	\xi_{\chi^{[2,1^{n-2}]}} &= \frac{-1}{n-1} \left( \binom{n}{2} + \binom{n-1}{2} -2 - 2\beta r -2\delta s  \right)\\
	\xi_{\chi^{[1^n]}} &=  -\left( 4\beta r + 2 \delta s + 3 - 2 \binom{n}{2} - 2\binom{n-1}{2} \right) 
	\end{align*}

Let $(\mathcal{P}'')$ be the polytope of $\mathbb{R}^2$
defined as follows.
\begin{align*}
	(\mathcal{P}'') \quad \left\{ \qquad
		\begin{aligned}
			& 2\beta x + \delta y +1 - \binom{n}{2} - \binom{n-1}{2} <0\\
			&\beta x + \delta y + 1 - \binom{n-1}{2} >0\\
			& y <0.
		\end{aligned}
	\right.
\end{align*}
Just as in the case when $n$ is even, this polytope is defined so that
the eigenvalues $\xi_{\chi^{[2^2,1^{n-4}]}}$,
$\xi_{\chi^{[2,1^{n-2}]}}$ and $\xi_{\chi^{[1^n]}}$ are all strictly
greater than $-1$.  The polytope $(\mathcal{P}'')$ is a triangle
(without the boundary) with coordinates
\[
\left( \frac{1}{\beta}\binom{n}{2}, \frac{-n}{\delta} \right),
\quad
\left( \frac{1}{\beta}\left(\binom{n-1}{2}-1\right),0 \right), 
\quad
\left( \frac{1}{2\beta} \left( \binom{n}{2} + \binom{n-1}{2}-1
  \right),0 \right).
\]

In the next lemma, we state that all weightings, except $\omega_4$, are
positive for any $(r,s) \in \mathcal{P}''$. We give the statement of this without a proof since it is
straightforward.

\begin{lem}
	For $i\in \{1,2,3,5,6\}$ and for $(r,s)\in \mathcal{P}''$, we have $\omega_i(r,s,0)\geq 0$.
\end{lem}

As in the even case, the eigenvalues of the weighted adjacency matrix
are in the required range whenever $(r,s)\in \mathcal{P}''$. We omit the
proof as it is identical to the case where $n$ is even.

\begin{lem}\label{pointwise-odd}
  Let $n\geq 13$ and odd. For $(r,s)\in \mathcal{P}''$, the eigenvalues
  of the weighted adjacency matrix $A$ defined in
  \eqref{w-mat-pointwise} are in $[-1,2\binom{n}{2}-1]$. Moreover, the
  only irreducible characters giving eigenvalue $-1$ are
  $\chi^{[n-1,1]},\ \chi^{[n-2,1^2]}$ and $\chi^{[n-2,2]}$.
\end{lem}
\begin{proof}
Lemma~\ref{low-dim} gives the eight irreducible representations of
$\sym(n)$ with degree less than $2\binom{n+1}{2}$. The polytope
$\mathcal{P}''$ is defined so that this result holds
for these eight irreducible representation.

Next assume $\chi$ is an irreducible character of $\sym(n)$ with $\chi(\id)
>2\binom{n+1}{2}$. On each of the conjugacy classes that we consider,
the value of any irreducible representation is bounded by $1$. 
Noting that $(r,s) \in \mathcal{P}''$ implies $\delta s \in (-n,0)$, we have the following bound.
\begin{align*}
|\xi_{\chi} (r,s,0)| &\leq \frac{1}{\chi(\id)} \left( \binom{n}{2}
          -1 -\beta r - \delta s \right) \left| \chi_{(n)} \right| +
        \beta r \left|  \chi_{(n-1,1)} \right|+ \left(1 -
          \binom{n-1}{2} + \beta r  \right) \left| \chi_{(n-2,2)} \right|\\
	&\quad +\left|\delta s\right| \left| \chi_{(n-3,3)}\right| + \left( \binom{n}{2} +
          \binom{n-1}{2} -1 - \beta r \right) \left| \chi_{(n-3,2,1)} \right|.\\
&\leq \frac{1}{2\binom{n+1}{2}}   \left( \binom{n}{2}
          -1 -\beta r +n \right) + \beta r  + 
  \left( 1 - \binom{n-1}{2} + \beta r  \right) \\
	&\quad +n  + \left( \binom{n}{2} +      \binom{n-1}{2} -1 - \beta r \right).\\
& \leq  \frac{1}{2\binom{n+1}{2}} \left( 2\binom{n}{2}-1 +2n\right) \\
& < \frac{n^2+n-1}{n^2+n} \\
&<1.
\end{align*}
Thus the eigenvalue corresponding to any irreducible representation
with degree greater than or equal to $2\binom{n+1}{2}$ is strictly
between $-1$ and $1$, so the result holds.
\end{proof}

The proof of Theorem~\ref{main-pointwise} follows from
Lemma~\ref{pointwise-ev} and Lemma~\ref{pointwise-odd} using the ratio
bound on $\Gamma_n$ with the weighted adjacency matrix $A$. 

\begin{proof}[Proof of Theorem~\ref{main-pointwise}]
  If $5\leq n \leq 12$, we use Sagemath \cite{sagemath} to confirm
  that the result holds. For $n\geq 13$, we prove the result by
  using Theorem~\ref{thm:wtRatio}. Therefore, if $n$ is even, then
  $\alpha(\Gamma_{\sym(n)}) \leq \alpha(X_1) = (n-2)!$, and similarly,
  if $n$ is odd, then $\alpha(\Gamma_{\sym(n)})\leq \alpha(X_2) =
  (n-2)!$.
\end{proof}

We finish this section by proving a conjecture of Godsil and Meagher (\cite{godsil2009new}, Conjecture~7.3).

\begin{cor}
  Let $n \geq 5$. If $\chi$ is the characteristic vector of a maximum $2$-pointwise
  intersecting family of $\sym(n)$, then $\chi \in V_{[n]} \oplus
  V_{[n-1,1]} \oplus V_{[n-2,2]} \oplus V_{[n-2,1^2]}$.
\end{cor}

\section{Further work}\label{further-work}

In this paper we construct weighted adjacency matrices for the
derangement graphs of $\sym(n)$ for two different actions. This work
proves that the conjectured lower bounds on $n$ for each action are
indeed the correct bound when $t=2$. It is also interesting that this work also
shows that there are infinitely many weighted adjacency matrices that
would work in the ratio bound. We leave
the reader with two open problems.

\begin{prob}
  In Theorem~\ref{main-thm} and Theorem~\ref{main-pointwise}, we only
  proved that any characterstic vector of a maximum $2$-setwise and
  $2$-pointwise intersecting family are in their respective
  permutation module. Characterize the maximum intersecting families
  for both type of intersections.
\end{prob}

\begin{prob}
  Prove that the $t$-setwise intersecting property of $\sym(n)$ holds
  for any $3\leq t\leq n$ and for $n\geq 2t+1$, prove that $\sym(n)$
  has the $t$-pointwise intersecting property.
\end{prob}

\bibliographystyle{plain}

\appendix
\section{ }
\begin{table}[H]
	\centering
	\begin{tabular}{ |c|c|c|c|c|c|c|c|} \hline
		& & $C_{(n-3,3)}$ &  $ C_{(n-4,3,1)}$ & $C_{(n)}$ & $C_{(n-1,1)}$ &\centering $C_{(n-2,2)}$ &$ C_{(n-3,2,1)}$ \\
		& & & & & & & \\
		Representation & Range of $k$ & & & & & &\\ \hline
		$\chi^{[n]}$ &  - &  $1$ & $1$ & $1$ & $1$ & $1$ &  $1$  \\ \hline
		$\chi^{[n-1,1]}$& - & $-1$& $0$ & $-1$ & $0$ & $-1$ & $0$ \\ \hline
		$\chi^{[n-2,2]}$& -& $0$& $-1$ & $0$ & $-1$ & $1$  & $0$ \\ \hline
		$\chi^{[n-2,1^2]}$& -& $1$& $0$& $1$ & $0$ & $0$ & $-1$ \\ \hline
		$\chi^{[n-3,3]}$& -& $1$& $1$& $0$ & $0$ & $-1$ & $0$  \\ \hline
		$\chi^{[n-3,2,1]}$& -& $-1$& $0$& $0$ & $0$ & $0$ & $1$  \\ \hline
		$\chi^{[n-3,1^3]}$& -& $0$& $-1$& $-1$ & $0$ & $0$ & $0$ \\ \hline
		$\chi^{[n-4,2^2]}$& -& $1$& $0$& $0$ & $0$  & $-1$ & $-1$ \\ \hline
		$\chi^{[n-4,2,1^2]}$& -& $0$& $1$& $0$ & $-1$  & $0$ & $0$ \\ \hline
		$\chi^{[n-5,2^2,1]}$& -& $0$& $-1$& $0$ & $0$  & $1$ & $0$ \\ \hline
		$\chi^{[n-6,2^3]}$& -& $-1$& $-1$& $0$ & $0$  & $0$ & $1$ \\ \hline
		$\chi^{[n-k-4,4,1^k]}$& $0 \leq k< n-8$ & $(-1)^{k+1}$& $0$& $0$ & $0$  & $0$ & $(-1)^{k+1}$ \\ \hline
		$\chi^{[n-k-5,5,1^k]}$&  $0 \leq k\leq n-10$ &  $0$ &  $(-1)^{k+1}$& $0$ & $0$  & $0$ & $0$ \\ \hline
		$\chi^{[n-k-5,3,2,1^k]}$& $0 \leq k\leq n-8$ & $(-1)^{k+3}$& $0$& $0$ & $0$  & $0$ & $0$ \\ \hline
		$\chi^{[n-k-6,2^3,1^k]}$& $0 < k\leq n-8$ & $(-1)^{k+3}$& $0$& $0$ & $0$  & $0$ & $(-1)^{k+4}$ \\ \hline
		$\chi^{[n-k-6,3^2,1^k]}$& $0 \leq k\leq n-9$ & $0$& $(-1)^{k+1}$& $0$ & $0$  & $0$ & $0$ \\ \hline
		$\chi^{[n-k-8,2^4,1^k]}$& $0 \leq k\leq n-10$ & $0$& $(-1)^{k+4}$& $0$ & $0$ & $0$ & $0$  \\ \hline
		$\chi^{[n-k,1^k]}$& $4 \leq k\leq n-5$ & $0$& $0$& $(-1)^{k}$ & $0$ & $0$ & $0$ \\ \hline
		$\chi^{[n-k-2,2,1^k]}$& $3 \leq k\leq n-6$ & $0$& $0$& $0$ & $(-1)^{k+1}$  & $0$ & $0$ \\ \hline
		$\chi^{[4^2,1^{n-8}]}$& -& $(-1)^{n-7}$& $(-1)^{n-6}$& $0$ & $0$  & $0$ & $(-1)^{n-7}$ \\ \hline
		$\chi^{[4,3,1^{n-7}]}$& -& $0$& $(-1)^{n-5}$& $0$ & $0$  & $(-1)^{n-6}$ & $0$ \\ \hline
		$\chi^{[4,2,1^{n-6}]}$& -& $0$& $(-1)^{n-4}$& $0$ & $(-1)^{n-5}$  & $0$ & $0$ \\ \hline
		$\chi^{[3^2,1^{n-6}]}$& -& $(-1)^{n-4}$& $0$& $0$ & $0$  & $(-1)^{n-5}$ & $(-1)^{n-6}$ \\ \hline
		$\chi^{[4,1^{n-4}]}$& -& $0$& $(-1)^{n-4}$& $(-1)^{n-4}$ & $0$  & $0$ & $0$ \\ \hline
		$\chi^{[3,2,1^{n-5}]}$& -& $(-1)^{n-3}$& $0$& $0$ & $(-1)^{n-4}$  & $0$ & $(-1)^{n-5}$ \\ \hline
		$\chi^{[2^3,1^{n-6}]}$& -& $(-1)^{n-2}$& $(-1)^{n-3}$& $0$ & $(-1)^{n-5}$  & $(-1)^{n-3}$ & $0$ \\ \hline
		$\chi^{[3,1^{n-3}]}$& -& $(-1)^{n-2}$& $0$& $(-1)^{n-3}$ & $0$  & $0$ & $(-1)^{n-4}$ \\ \hline
		$\chi^{[2^2,1^{n-4}]}$& -& $0$& $(-1)^{n-4}$ & $0$ & $(-1)^{n-3}$  & $(-1)^{n-4}$ & $0$ \\ \hline
		$\chi^{[2,1^{n-2}]}$& -& $(-1)^{n-1}$& $0$& $(-1)^{n-2}$ & $0$  & $(-1)^{n-3}$ & $0$ \\ \hline
		$\chi^{[1^n]}$& - &  $(-1)^{n-2}$ &  $(-1)^{n-1}$ & $(-1)^{n-1}$ & $(-1)^{n-2}$  & $(-1)^{n-2}$ & $(-1)^{n-3}$ \\ \hline
	\end{tabular}
	\caption{Values of irreducible characters on conjugacy classes
          $C_{(n-3,3)}$, $ C_{(n-4,3,1)}$, $C_{(n)}$, $C_{(n-1,1)}$,
          $C_{(n-2,2)}$ and $ C_{(n-3,2,1)}$.}\label{char-val-pointwise}
\end{table}
\end{document}